\documentclass{amsart}

\newtheorem{theorem}{Theorem}[section]
\newtheorem{lemma}[theorem]{Lemma}

\newtheorem{corollary}[theorem]{Corollary}

\theoremstyle{definition}

\newtheorem{remark}[theorem]{Remark}

\usepackage{amscd,amssymb}
\usepackage{ytableau}

\begin{document}

\title[Another proof of the Nowicki conjecture]
{Another proof of the Nowicki Conjecture}
\author[Vesselin Drensky]
{Vesselin Drensky}
\address{Institute of Mathematics and Informatics,
Bulgarian Academy of Sciences,
1113 Sofia, Bulgaria}
\email{drensky@math.bas.bg}


\begin{abstract}
Let $K[X_d,Y_d]=K[x_1,\ldots,x_d,y_1,\ldots,y_d]$ be the polynomial algebra in $2d$ variables over a field $K$ of characteristic 0
and let $\delta$ be the derivation of $K[X_d,Y_d]$ defined by
$\delta(y_i)=x_i$, $\delta(x_i)=0$, $i=1,\ldots,d$.
In 1994 Nowicki conjectured that the algebra $K[X_d,Y_d]^{\delta}$ of constants of $\delta$
is generated by $X_d$ and $x_iy_j-y_ix_j$ for all $1\leq i<j\leq d$.
The affirmative answer was given by several authors using different ideas.
In the present paper we give another proof of the conjecture based on representation theory of
the general linear group $GL_2(K)$.
\end{abstract}

\maketitle

\section{Introduction}
The linear operator $\delta$ of an algebra $A$ over a field $K$ is a derivation if
it satisfies the Leibniz rule
\[
\delta(a_1a_2)=\delta(a_1)a_2+a_1\delta(a_2),\quad a_1,a_2\in A.
\]
The kernel of $\delta$ is called the algebra of constants of $\delta$ and is denoted by $A^{\delta}$.
In the sequel $K$ is a field of characteristic 0. When $A=K[Z_m]=K[z_1,\ldots,z_m]$ is the algebra of
polynomials in $m$ variables the derivation $\delta$ is called Weitzenb\"ock if
it acts as a nilpotent linear operator on the vector space $KZ_m$ with basis $Z_m=\{z_1,\ldots,z_m\}$.
The classical theorem of Weitzenb\"ock \cite{Wei} in 1932 states that in this case $K[Z_m]^{\delta}$ is finitely generated.
The algebra $K[Z_m]^{\delta}$ coincides with the algebra of invariants $K[Z_m]^{(K,+)}$, where
the additive group $(K,+)$ of the field $K$ is
embedded as a subgroup into the unitriangular group $UT_m(K)$ acting as $\{\exp(\alpha\delta)\mid\alpha\in K\}$.
Hence the finitely generation of $K[Z_m]^{\delta}$ is equivalent to a theorem of classical invariant theory.
A modern geometric proof of the Weitzenb\"ock theorem in this spirit is given by Seshadri \cite{S}.
A translation in an algebraic language of this proof is given by Tyc \cite{T}.
For more information on Weitzenb\"ock derivations
one can see the books by Nowicki \cite[Section 6.2]{N},  Derksen and Kemper \cite[Chapter 2]{DK}, and Dolgachev \cite[Section 4.2]{D}.

In the special case of the polynomial algebra $K[X_d,Y_d]$ in $2d$ variables
$X_d=\{x_1,\ldots,x_d\}$ and $Y_d=\{y_1,\ldots,y_d\}$ when the Weitzenb\"ock derivation $\delta$ acts by
\begin{equation}\label{action of delta}
\delta(x_i)=0,\quad\delta(y_i)=x_i,\quad i=1,\ldots,d,
\end{equation}
Nowicki \cite{N} conjectured in 1994 that
$K[X_d,Y_d]^{\delta}$ is generated by $X_d$ and the determinants
\begin{equation}\label{determinants}
u_{ij}=\left\vert\begin{matrix}x_i&x_j\\
y_i&y_j\\
\end{matrix}\right\vert=x_iy_j-x_jy_i,\quad 1\leq i<j\leq d.
\end{equation}
There are several proofs based on different methods confirming the Nowicki conjecture:
by Khoury \cite{K1, K2},  Bedratyuk \cite{B},
the author and Makar-Limanov \cite{DML}, Kuroda \cite{Ku}. There are unpublished proofs by Derksen and Panyushev.
As Kuoda mentions in his paper \cite{Ku} Goto, Hayasaka, Kurano, and Nakamura \cite{GHKN} and Miyazaki \cite{M}
determined sets of generators for algebras of invariants with $K[X_d,Y_d]^{\delta}$ included in the list.

In the present paper we give a new proof using easy arguments from representation theory of the general linear group $GL_2(K)$.
Our proof is inspired by our paper with Gupta \cite{DG} devoted to the noncommutative version of
Weitzenb\"ock derivations.

\section{Preliminaries}

Let $V$ be a vector space with basis $\{x,y\}$ with the canonical action of the general linear group $GL_2(K)$:
\begin{equation}\label{canonical action of GL}
g(x)=\gamma_{11}x+\gamma_{21}y,\quad
g(y)=\gamma_{12}x+\gamma_{22}y, \text{ where }g=\left(\begin{matrix}\gamma_{11}&\gamma_{12}\\
\gamma_{21}&\gamma_{22}\\
\end{matrix}\right)\in GL_2(K).
\end{equation}
For a background on representation theory of the general linear group $GL_m(K)$ see, e.g.,
the books by Weyl \cite[Chapter 4]{Wey} or James and Kerber \cite[Chapter 8]{JK}.
We shall summarize the necessary facts for $m=2$ only.
The polynomial representations of $GL_2(K)$ are completely reducible and their irreducible components
$W(\lambda)$ are indexed by partitions $\lambda=(\lambda_1,\lambda_2)$. If $\lambda$ is a partition of $n$
(notation $\lambda\vdash n$), then
$W(\lambda)$ can be realized as a $GL_2(K)$-submodule of the $n$-th tensor degree $V^{\otimes_n}$
equipped with the diagonal action of $GL_2(K)$
\[
g(v_1\otimes\cdots\otimes v_n)=g(v_1)\otimes\cdots\otimes g(v_n),\quad v_i\in V,\quad g\in GL_2(K).
\]
As a vector space $V^{\otimes_n}$ is ${\mathbb N}_0^2$-graded and the homogeneous component of degree $(n_1,n_2)$, $n_1+n_2=n$,
is spanned on the tensors $z_{i_1}\otimes\cdots\otimes z_{i_n}$, $z_i=x,y$, of degree $n_1$ and $n_2$ in $x$ and $y$, respectively.
Then $W(\lambda)$ has a basis of homogeneous elements
\begin{equation}\label{basis of W}
\{w_0,w_1,\ldots,w_{\lambda_1-\lambda_2}\},\quad \deg w_i=(\lambda_1-i,\lambda_2+i),\quad i=0,1,\ldots,\lambda_1-\lambda_2.
\end{equation}
The element $w_0=w(\lambda)\in W(\lambda)$ which is homogeneous of degree $\lambda$ is called the highest weight vector of $W(\lambda)$.
One typical element $w(\lambda)^{(1)}\in W(\lambda)^{(1)}\subset V^{\otimes_n}$, $W(\lambda)^{(1)}\cong W(\lambda)$, is
\begin{equation}\label{special HWV}
w(\lambda)^{(1)}=\underbrace{(x\otimes y-y\otimes x)\otimes \cdots\otimes (x\otimes y-y\otimes x)}_{\lambda_2\text{ times}}
\otimes\underbrace{x\otimes\cdots\otimes x}_{\lambda_1-\lambda_2\text{ times}}.
\end{equation}
Here the skew-symmetric sums $(x\otimes y-y\otimes x)$ appear in positions $(1,2), (3,4),\ldots$, $(2\lambda_2-1,2\lambda_2)$.
The symmetric group $S_n$ acts from the right on $V^{\otimes_n}$ by place permutation
\[
(v_1\otimes\cdots\otimes v_n)\sigma^{-1}=v_{\sigma(1)}\otimes\cdots\otimes v_{\sigma(n)},\quad v_i\in V,\quad \sigma\in S_n.
\]
Then every highest weight vector $w(\lambda)\in W(\lambda)\subset V^{\otimes_n}$ is of the form
\begin{equation}\label{general HWV}
w(\lambda)=w(\lambda)^{(1)}\sum_{\sigma\in S_n}\alpha_{\sigma}\sigma^{-1},\quad\alpha_{\sigma}\in K.
\end{equation}
Clearly, the skew-symmetries in $w(\lambda)^{(1)}\sigma$ are in positions
\[
(\sigma(1),\sigma(2)), (\sigma(3),\sigma(4)),\ldots, (\sigma(2\lambda_2-1),\sigma(2\lambda_2)).
\]

\begin{remark}
Since $W(\lambda)$, $\lambda\vdash n$, participates in $V^{\otimes_n}$ with multiplicity equal to the number of standard tableaux
of shape $[\lambda]$, by \cite[Proposition 0.1]{DR}
we may choose a basis of the vector space of highest weight vectors $w(\lambda)\in V^{\otimes_n}$ consisting of all
$w(\lambda)^{(1)}\sigma^{-1}$ such that the tableau
\begin{center}
\ytableausetup
{mathmode, boxsize=5em}
\begin{ytableau}\sigma(1)&\sigma(3)&\cdots&\sigma(2\lambda_2-1)&\sigma(2\lambda_2+1)&\cdots&\sigma(n)\\
\sigma(2)&\sigma(4)&\cdots&\sigma(2\lambda_2)
\end{ytableau}
\end{center}
is standard, i.e.,
\[
\sigma(1)<\sigma(2),\sigma(3)<\sigma(4),\ldots,\sigma(2\lambda_2-1)<\sigma(2\lambda_2),
\]
\[
\sigma(1)<\sigma(3)<\cdots <\sigma(2\lambda_2-1)<\sigma(2\lambda_2+1)<\sigma(2\lambda_2+2)<\cdots<\sigma(n),
\]
\[
\sigma(2)<\sigma(4)<\cdots< \sigma(2\lambda_2).
\]
\end{remark}

The highest weight vector of $W(\lambda)\subset V^{\otimes_n}$ can be characterized in the following way,
see \cite[Lemma 1.1]{BD}.

\begin{lemma}\label{characterization of HWV} Let $\Delta$ be the derivation of the tensor algebra
\[
T(V)=\bigoplus_{n\geq 0}V^{\otimes_n}=K\oplus V\oplus (V\otimes V)\oplus (V\otimes V\otimes V)\oplus\cdots
\]
defined by $\Delta(x)=0$, $\Delta(y)=x$. Then the homogeneous element $w\not=0$ of degree $\lambda\vdash n$ is
a highest weight vector of some $W(\lambda)\subset V^{\otimes_n}$ if and only if $\Delta(w)=0$.
\end{lemma}

\begin{remark}\label{action on basis of W}
Up to a nonzero multiplicative constant the derivation $\Delta$ sends the element $w_i$ from (\ref{basis of W}) to $w_{i-1}$,
$i=1,\ldots,\lambda_1-\lambda_2$, and $\Delta(w_0)=0$.
\end{remark}

Let $n=(n_1,\ldots,n_d)$ be a $d$-tuple of nonnegative integers and let $\vert n\vert=n_1+\cdots+n_d$.
Consider the vector spaces $V_1,\ldots,V_d$ with bases $\{x_1,y_1\},\ldots,\{x_d,y_d\}$, respectively,
and the canonical action of $GL_2(K)$ as in (\ref{canonical action of GL}) on them. Clearly,
the tensor products $V^{(n)}=V_1^{\otimes_{n_1}}\otimes\cdots\otimes V_d^{\otimes_{n_d}}$ and $V^{\otimes_{\vert n\vert}}$ are
isomorphic as $GL_2(K)$-modules. As in the case of $V^{\otimes_n}\subset T(V)$ we define an ${\mathbb N}_0^2$-grading on
$V^{(n)}$ which counts the number of entries of $X_d$ and $Y_d$, respectively.
Let $\lambda\vdash\vert n\vert$. If at the first $\lambda_2$ couples of positions $(1,2),(3,4),\ldots,(2\lambda_2-1,2\lambda_2)$
in the tensor product $V^{(n)}$ we have $(V_{i_1},V_{j_1}),(V_{i_2},V_{j_2}),\ldots,(V_{i_{\lambda_2}},V_{j_{\lambda_2}})$,
and the positions left are $k_1,\ldots,k_{\lambda_1-\lambda_2}$,
then the analogue of the equation (\ref{special HWV}) is
\begin{equation}\label{version of special HWV}
(x_{i_1}\otimes y_{j_1}-y_{i_1}\otimes x_{j_1})\otimes
\cdots\otimes (x_{i_{\lambda_2}}\otimes y_{j_{\lambda_2}}-y_{i_{\lambda_2}}\otimes x_{j_{\lambda_2}})
\otimes x_{k_1}\otimes\cdots\otimes x_{k_{\lambda_1-\lambda_2}}.
\end{equation}
The equation (\ref{general HWV}) also can be restated in a similar way.
As a consequence of Lemma \ref{characterization of HWV} we obtain:

\begin{corollary}\label{HWV for tensor products}
Let $\delta$ be the derivation of the tensor algebra $T(V_1\oplus\cdots\oplus V_d)$ defined by
{\rm (\ref{action of delta})}. Then a homogeneous element $w\in V^{(n)}=V_1^{\otimes_{n_1}}\otimes\cdots\otimes V_d^{\otimes_{n_d}}$
of degree $\lambda\vdash\vert n\vert$ is a highest weight vector of a submodule $W(\lambda)$ of
$V^{(n)}$ if and only if $\delta(w)=0$.
\end{corollary}

\section{The main result}

We are ready to present our proof of the Nowicki conjecture.

\begin{theorem}
Let $K$ be a field of characteristic $0$ and let $\delta$ be the derivation of the polynomial algebra
$K[X_d,Y_d]$ defined by {\rm (\ref{action of delta})}. Then the algebra of constants $K[X_d,Y_d]^{\delta}$
is generated by $X_d$ and the determinants {\rm (\ref{determinants})}.
\end{theorem}

\begin{proof}
The algebra $K[X_d,Y_d]$ has a canonical ${\mathbb N}_0^d$-grading. The homogeneous component $K[X_d,Y_d]^{(n)}$ of degree $n=(n_1,\ldots,n_d)$
is spanned by the monomials which are of degree $n_i$ in $x_i$ and $y_i$, $i=1,\ldots,d$.
It follows from the definition of $\delta$ that $\delta(K[X_d,Y_d]^{(n)})\subset K[X_d,Y_d]^{(n)}$.
Hence we shall prove the theorem if we show that each component $(K[X_d,Y_d]^{(n)})^{\delta}$ is spanned on the products
\begin{equation}\label{product of X and U}
X_d^pU_d^q=x_1^{p_1}\cdots x_d^{p_d}\prod_{1\leq i<j\leq d}u_{ij}^{q_{ij}}
\end{equation}
of degree $n$. As a $GL_2(K)$-module $K[X_d,Y_d]^{(n)}$ is isomorphic to the symmetric tensor power of
$n_1$ copies of $V_1$, $n_2$-copies of $V_2$, $\ldots$, $n_d$ copies of $V_d$. Hence it is a homomorphic image of
$V^{(n)}=V_1^{\otimes_{n_1}}\otimes\cdots\otimes V_d^{\otimes_{n_d}}$.
The action of $\delta$ on $V^{(n)}$ induces the canonical action on $K[X_d,Y_d]^{(n)}$.
Therefore the vector space of the highest weight vectors of $K[X_d,Y_d]^{(n)}$
is an image of the vector space of the highest weight vectors of $V^{(n)}$
and by Lemma \ref{characterization of HWV} and Remark \ref{action on basis of W}
coincides with $(K[X_d,Y_d]^{(n)})^{\delta}$.
The highest weight vectors of $V^{(n)}$ are linear combinations of the products (\ref{version of special HWV})
with the property that
\[
\{i_a,j_a,k_b\mid a=1,2,\ldots,\lambda_2,b=1,2,\ldots,\lambda_1-\lambda_2\}=\{1,2,\ldots,\vert n\vert\}.
\]
Obviously the image of the element (\ref{version of special HWV}) in $K[X_d,Y_d]^{(n)}$ is
\[
u_{i_1j_1}\cdots u_{i_{\lambda_2},j_{\lambda_2}}x_{k_1}\cdots x_{k_{\lambda_1-\lambda_2}}
\]
Replacing $u_{i_aj_a}$ with $u_{j_ai_a}$ if $i_a>j_a$,
we obtain that $(K[X_d,Y_d]^{(n)})^{\delta}$ is spanned on the products (\ref{product of X and U})
which completes the proof.
\end{proof}

\end{document}